\documentclass[reqno,12pt]{article}

\usepackage{a4wide}
\usepackage{amsmath} 
\usepackage{amssymb}
\usepackage{amsthm}
\usepackage[utf8]{inputenc} 
\usepackage{graphicx} 
\usepackage[english, russian]{babel}
\usepackage{xcolor}
\usepackage[normalem]{ulem}

\numberwithin{equation}{section}

\newtheorem{theo}{Theorem}
\newtheorem{conj}{Conjecture}
\newtheorem{coro}{Corollary}
\newtheorem{prop}{Proposition}

\newtheorem{lem}{Lemma}
\newtheorem{defi}{Definition}

\newtheorem{thmx}{Theorem}

\theoremstyle{remark}
\newtheorem{Remark}{Remark}

\newcommand{\N}{\mathbb{N}}
\newcommand{\Z}{\mathbb{Z}}
\newcommand{\Q}{\mathbb{Q}}
\newcommand{\R}{\mathbb{R}}
\renewcommand{\C}{\mathbb{C}}

\newcommand{\Qbar}{\overline{\mathbb Q}}

\newcommand{\etoile}{^*}

\newcommand{\eps}{\varepsilon}

\renewcommand{\Im}{{\rm Im}}

\newcommand{\calP}{\mathcal P}

\newcommand{\trdeg}{{\rm trdeg}}
\newcommand{\rk}{{\rm rk}}

\newcommand{\tantie}{\foreignlanguage{russian}{\CYREREV}}

\newcommand{\antie}{\text{\tantie}}
\newcommand{\mixed}{{\bf M}_G}
\newcommand{\GG}{{\bf G}}
\newcommand{\VV}{{\bf V}}
\newcommand{\EE}{{\bf E}}
\newcommand{\bfS}{{\bf S}}

\newcommand{\fgo}{\mathfrak{f}} 

\begin{document}

 \selectlanguage{english}

\title{
Relations between values of arithmetic Gevrey series, 
and applications to values of the Gamma function
}
\date\today
\author{S. Fischler and T. Rivoal}
\maketitle

\begin{abstract} We investigate the relations between the rings {\bf E}, {\bf G} 
 and  {\bf D} 
of values taken at algebraic points by arithmetic Gevrey series of order either $-1$ ($E$-functions),  $0$ (analytic continuations of $G$-functions) or $1$ (renormalization of divergent series solutions at $\infty$ of $E$-operators) respectively.
We prove in particular that any element of {\bf  G} can be written as multivariate polynomial with algebraic coefficients in elements of {\bf E} and {\bf D}, and is the limit at infinity of some $E$-function along some direction. This prompts to defining and  studying the notion of mixed functions, which generalizes simultaneously $E$-functions and arithmetic Gevrey series of order~1. Using natural conjectures for arithmetic Gevrey series of order 1 and mixed functions (which are analogues of a theorem of Andr\'e and Beukers for $E$-functions) and the conjecture ${\bf D}\cap{\bf E}=\Qbar$ (but not necessarily all these conjectures at the same time), we deduce a number of interesting Diophantine results such as an analogue for mixed functions of Beukers' linear independence theorem for values of $E$-functions,  the transcendance of the values of the Gamma function and its derivatives at all non-integral algebraic numbers, the transcendance of Gompertz constant as well as the fact that Euler's constant is not in {\bf E}. 
\end{abstract}

%\medskip

\section{Introduction} \label{sec:intro}

A power series $\sum_{n=0}^\infty \frac{a_n}{n!} x^n \in \Qbar[[x]]$ is said to be an $E$-function when it is solution of a linear differential equation over $\Qbar(x)$ (holonomic), and $\vert \sigma(a_n)\vert$ (for any $\sigma\in \textup{Gal}(\Qbar/\Qbar)$) and the least common denominator of $a_0, a_1, \ldots, a_n$ grow at most exponentially in $n$. They were defined and studied by Siegel in 1929, who also defined the class of $G$-functions: a power series $\sum_{n=0}^\infty a_n x^n \in \Qbar[[x]]$ is said to be a $G$-function when $\sum_{n=0}^\infty \frac{a_n}{n!} x^n$ is an $E$-function. In this case,  $\sum_{n=0}^\infty n!a_n z^n\in \Qbar[[z]]$  is called an \antie-function, following the terminology introduced by Andr\'e in \cite{andre}. $E$-functions are entire, while $G$-functions have a positive radius of convergence, which is finite except for polynomials. Here and below, we see $\Qbar$ as embedded into $\C$. Following André again, $E$-functions, $G$-functions and \tantie-fonctions are exactly arithmetic Gevrey series of  order $s=-1,0,1$ respectively. Actually André defines arithmetic Gevrey series of any order $s\in\Q$, but the set of values at algebraic points is the same for a given $s\neq 0$ as for $s/|s|$ using \cite[Corollaire 1.3.2]{andre}.

{\antie}-functions are divergent series, unless they are polynomials. 
Given an \antie-function $\fgo$ and any $\theta\in\R$, except finitely many values mod $2\pi$ (namely anti-Stokes directions of $\fgo$), one can perform Ramis' 1-summation of $\fgo(1/z)$ in the direction $\theta$, which coincides in this setting with Borel-Laplace summation (see \cite{Ramis} or \cite{ateo}). This provides a function denoted by $\fgo_\theta(1/z)$, holomorphic on the open subset of $\C$ consisting in all $z\neq 0$ such that $\theta-\frac\pi 2 -\eps < \arg z <  \theta+\frac\pi 2 +\eps$ for some $\eps>0$, of which $\fgo(1/z)$ is the asymptotic expansion in this sector (called a large sector bisected by $\theta$). Of course $\fgo(1/z)$ can be extended further by analytic continuation, but this asymptotic expansion may no longer be valid. When an  \antie-function  is denoted by $\fgo_j$, we shall denote by $\fgo_{j,\theta}$ or $\fgo_{j;\theta}$ its 1-summation and we always assume (implicitly or explicitly) that $\theta $ is not an anti-Stokes direction.

In \cite{gvalues}, \cite{ateo} and \cite[\S 4.3]{fuchsien}, we have studied the sets {\bf G},  {\bf E} and {\bf D} defined respectively as the sets of all the values taken  by all (analytic continuations of) $G$-functions at algebraic points, of  all the values taken by all  $E$-functions at algebraic points and  of all values 
$\fgo_\theta(1)$ where $\fgo$ is an  \antie-function ($\theta=0$ if it is not an anti-Stokes direction, and $\theta>0$ is very small otherwise.) These three sets are countable sub-rings of $\mathbb C$ that all contain $\Qbar$; conjecturally, they are related to the set of periods and exponential periods, see \S\ref{sec:proofthmpolyGEantiE}. (The ring {\bf D} is denoted by {\bf \tantie} in \cite{fuchsien}.)

We shall prove the following result in \S\ref{sec:proofthmpolyGEantiE}.

\begin{theo}\label{theo:GpolyEantiE} Every element of ${\bf G}$ can be written as a multivariate polynomial (with coefficients in $\Qbar$)  in elements of 
{\bf E} and {\bf D}.

Moreover, ${\bf G}$ coincides with the set of all convergent integrals $\int_0^{\infty }F(x)dx$ where $F$ is an $E$-function, or equivalently with the set of all finite limits of $E$-functions at $\infty$ along some direction.
\end{theo}
Above, a convergent integral $\int_0^{\infty}F(x)dx$ means a finite limit of the $E$-function $\int_0^zF(x)dx$ as $z\to \infty$ along some direction; this explains the equivalence of both statements.

We refer to  Eq.~\eqref{eq:log(2)EantiE} in \S\ref{sec:proofthmpolyGEantiE} for an expression of $\log(2)$ as a polynomial in elements in {\bf E} and~{\bf D}; the number $\pi$ could be similarly expressed by considering $z$ and $iz$ instead of $z$ and $2z$ there.  
Examples of the last statement are the identities (see \cite{zucker} for the second one):
$$
\int_0^{+\infty} \frac{\sin(x)}x dx=\frac{\pi}{2} \quad \textup{and} \quad \int_0^{+\infty}J_0(ix)e^{-3x}dx=\frac{\sqrt{6}}{96\pi^3}
\Gamma\Big(\frac{1}{24}\Big)\Gamma\Big(\frac{5}{24}\Big)\Gamma\Big(\frac{7}{24}\Big)\Gamma\Big(\frac{11}{24}\Big).
$$

%\medskip

It is notoriously difficult to prove/disprove that a given element of {\bf G} is transcendental; it is known that a Siegel-Shidlovskii type theorem for $G$-functions can not hold {\em mutatis mutandis}. Theorem~\ref{theo:GpolyEantiE} suggests that an alternative approach to the study of the Diophantine properties of elements of {\bf G} can be through a better understanding of joint study of the elements of {\bf E} and {\bf D}, modulo certain conjectures to begin with. Our applications will not be immediately directed to the elements of {\bf G} but rather to the understanding of the (absence of) relations between the elements of {\bf E} and {\bf D}.

%\medskip

 It seems natural (see \cite[p.~37]{ateo}) to conjecture that $\EE \; \cap \;{\bf G}=\Qbar$, and also that ${\bf G}\;\cap \;{\bf D}=\Qbar$, though both properties seem currently out of reach.  In this paper, we suggest (see \S \ref{ssec:mixedfunctions}) a possible approach towards the following analogous conjecture.

\begin{conj} \label{conj3} We have ${\bf E} \;\cap\;{ \bf D} =\Qbar$.
\end{conj}

In \S \ref{ssec:mixedfunctions} we shall make a functional conjecture, namely Conjecture~\ref{conj1}, that implies  Conjecture~\ref{conj3}. We also prove that  Conjecture~\ref{conj3} has very important  consequences, as the following result shows. 

\begin{theo}\label{prop3} 
 Assume that Conjecture \ref{conj3} holds. Then $\Gamma^{(s)}(a)$  is a transcendental number for any rational number $a>0$ and any integer $s\geq 0$, except of course if $s=0$ and $a\in\N$.
\end{theo}

%\medskip

One of the aims of this paper is to show that combining \antie- and $E$-functions may lead to very important results in transcendental number theory. Let us recall now briefly the main known results on $E$-functions.

%\medskip

Point $(i)$ in the following result is due to André  \cite{andre2} for $E$-functions with rational Taylor coefficients, and to Beukers  \cite{beukers} in the general case. 
 Andr\'e used this property to obtain a new proof of the Siegel-Shidlovskii Theorem, and Beukers   to prove an optimal refinement of this theorem (namely, $(ii)$  
 below).

\begin{thmx} \label{theoe} 
$(i)$ $[$André, Beukers$]$ If an $E$-function $F(z)$ is such that $F(1)=0$, then $\frac{F(z)}{z-1}$ is an $E$-function.

$(ii)$ $[$Beukers$]$
Let $\underline{F}(z):={}^t(f_1(z), \ldots, f_n(z))$ be a vector of $E$-functions  solution of a differential system $\underline{F}'(z)=A(z)\underline{F}(z)$ for some matrix $A(z)\in M_n(\Qbar(z))$. 

Let $\xi\in\Qbar\etoile$ which is not a pole of a coefficient of $A$. Let $P\in\Qbar[X_1,\ldots,X_n]$ be a homogeneous polynomial such that 
$$P(f_{1}(\xi), \ldots, f_{n}(\xi))=0.$$
Then there exists $Q\in \Qbar[Z,X_1,\ldots,X_n]$, homogeneous in the $X_i$, such that 
$$Q(z, f_1(z), \ldots, f_n(z))=0 \mbox{ identically  and } P(X_1,\ldots,X_n)=Q(\xi, X_1,\ldots,X_n).$$
In particular, we have
$$
\trdeg_{\Qbar}(f_{1}(\xi), \ldots, f_{n}(\xi)) = 
\trdeg_{\Qbar(z)}( f_{1 }(z), \ldots, f_{n}(z)).
$$
\end{thmx}
The Siegel-Shidlovskii Theorem itself is the final statement about equality of transcendence degrees. 

In this paper we state conjectural analogues of these results, involving \antie-functions.  The principal difficulty is that these functions are divergent power series, and the exact  analogue of Theorem \ref{theoe} is meaningless. André discussed  the situation in \cite{andre2} and even though he did not formulate exactly the following conjecture, it seems plausible to us. From it, we will show how to deduce an analogue of  the Siegel-Shidlovskii theorem for \antie-functions. Ferguson \cite[p.~171, Conjecture  1]{ferguson} essentially stated this conjecture when $\fgo(z)$ has {\em rational} coefficients and when $\theta=0$.

\begin{conj} \label{conjantie} Let $\fgo(z)$ be an \antie-function and $\theta\in (-\pi/2,\pi/2)$ be such that $\fgo_\theta(1)=0$. Then $\frac{\fgo(z)}{z-1}$ is an  \antie-function.
\end{conj}

In other words, the conclusion of this conjecture asserts that $\frac{z}{1-z} \fgo(1/z)$ is an \antie-function in $1/z$; this is equivalent to $\frac{\fgo(1/z)}{z-1}$ being an  \antie-function in $1/z$ (since we have $ \frac{\fgo(1/z)}{z-1} = O(1/z)$ unconditionally as $|z|\to\infty$). 

%\medskip

Following Beukers' proof \cite{beukers} yields the following result (see \cite[\S 4.6]{Andrebordeaux} for a related conjecture).

\begin{theo}\label{theoantie} 
Assume that Conjecture \ref{conjantie}  holds.

Let $\underline{\fgo}(z):={}^t(\fgo_1(z), \ldots, \fgo_n(z))$ be a vector of $\antie$-functions  solution of a differential system $\underline{\fgo}'(z)=A(z)\underline{\fgo}(z)$ for some matrix $A(z)\in M_n(\Qbar(z))$. Let $\xi\in\Qbar\etoile$ and $\theta\in (\arg(\xi)-\pi/2,\arg(\xi)+\pi/2)$ ; assume that $\xi$ is not a pole of a coefficient of $A$, and that $\theta$ is anti-Stokes for none of the $\fgo_j$. 

Let $P\in\Qbar[X_1,\ldots,X_n]$ be a homogeneous polynomial such that 
$$P(\fgo_{1,\theta}(1/\xi), \ldots, \fgo_{n,\theta}(1/\xi))=0.$$
Then there exists $Q\in \Qbar[Z,X_1,\ldots,X_n]$, homogeneous in the $X_i$, such that 
$$Q(z, \fgo_1(z), \ldots, \fgo_n(z))=0 \mbox{ identically  and } P(X_1,\ldots,X_n)=Q(1/\xi, X_1,\ldots,X_n).$$
In particular, we have
$$\trdeg_{\Qbar}( \fgo_{1,\theta}(1/\xi), \ldots, \fgo_{n,\theta}(1/\xi) ) = 
\trdeg_{\Qbar(z)}( \fgo_{1 }(z), \ldots, \fgo_{n}(z)). $$
\end{theo}

As an application of Theorem \ref{theoantie}, we shall prove the following corollary. Note that under his weaker version of Conjecture~\ref{conjantie}, Ferguson \cite[p.~171, Theorem 2]{ferguson} proved that Gompertz's constant is an irrational number.

\begin{coro} \label{coro:0701} 
Assume that Conjecture \ref{conjantie} holds. Then for any $\alpha\in \Qbar$, $\alpha>0$, and any $s\in \mathbb Q \setminus \mathbb Z_{\ge 0}$, the number $\int_0^\infty (t+\alpha)^s e^{-t}dt$ is a transcendental number.

In particular, Gompertz's constant $\delta:=\int_0^\infty e^{-t}/(t+1) dt$ is a transcendental number.
\end{coro}

In this text we suggest an approach towards Conjecture \ref{conj3}, based on the new notion of {\em mixed functions} which enables one to consider $E$- and \tantie-functions at the same time. In particular we shall state a conjecture about such functions, namely Conjecture~\ref{conj1} in \S\ref{ssec:mixedfunctions}, which implies both Conjecture \ref{conj3} and Conjecture~\ref{conjantie}. The following result is a motivation for this approach.

\begin{prop}\label{proppasdansE}
 Assume that both Conjectures \ref{conj3} and \ref{conjantie} hold. Then 
neither Euler's constant $\gamma:=-\Gamma'(1)$ nor $\Gamma(a)$ (with $a\in \mathbb Q^+\setminus \mathbb N$) are in {\bf E}.
\end{prop}

It is likely that none of these numbers   is  in {\bf G}, but (as far as we know) there is no ``functional'' conjecture like Conjecture \ref{conj1} that implies this. It is also likely that none is in  {\bf D} as well, but we don't know if this can be deduced from Conjecture \ref{conj1}.

\bigskip

The structure of this paper is as follows. In \S \ref{ssec:mixedfunctions} we define and study mixed functions, a combination of $E$- and \antie-functions. Then in \S \ref{sec:proofthmpolyGEantiE} we express any value of  a $G$-function as a polynomial in values of $E$- and \antie-functions, thereby proving Theorem \ref{theo:GpolyEantiE}.
We study derivatives of the $\Gamma $ function at rational points  in \S \ref{sec:consequences}, and prove  Theorem~\ref{prop3} and Proposition~\ref{proppasdansE}. At last, \S \ref{sec5} is devoted to adapting  Beukers' method to our setting: this approach yields  Theorem \ref{theoantie} and Corollary \ref{coro:0701}.

\section{Mixed  functions}\label{ssec:mixedfunctions}

\subsection{Definition and properties}

In view of Theorem \ref{theo:GpolyEantiE}, it is natural to study polynomials in $E$- and \tantie-functions. We can prove a Diophantine result that combines both Theorems \ref{theoe}$(ii)$ and \ref{theoantie} but under a very complicated polynomial  generalization of Conjecture \ref{conjantie}. We opt here for a different approach to mixing $E$- and \tantie-functions for which very interesting Diophantine consequences can be deduced from a very easy to state conjecture which is more in the spirit of Conjecture \ref{conjantie}. We refer to \S \ref{ssec:proofprop1} for proofs of all properties stated in this section (including Lemma~\ref{lemunicite} and Proposition \ref{prop1}), except Theorem \ref{theomixte}.

\begin{defi} \label{def:mixedfn}
We call {\em mixed (arithmetic Gevrey) function} any formal power series
$$\sum_{n\in\Z} a_n z^n$$
such that $\sum_{n\geq 0} a_n z^n$ is an $E$-function in $z$, and $\sum_{n\geq 1} a_{-n} z^{-n}$ is an \antie-function in $1/z$. 
\end{defi}
In other words, a mixed function is defined as a formal sum $\Psi(z)=F(z)+ \fgo (1/z)$ where $F$ is an $E$-function and $\fgo$ is an \antie-function. In particular, such a function is zero if, and only if, both $F$ and $\fgo$ are constants such that $F+\fgo=0$; obviously, $F$ and $\fgo$ are uniquely determined by $\Psi$ upon assuming (for instance) that $\fgo(0)=0$. The set of mixed functions is a $\Qbar$-vector space stable under multiplication by $z^n $ for any $n\in \Z$.  Unless $\fgo(z)$ is a polynomial, such a function $\Psi(z)=F(z)+ \fgo (1/z)$ is purely formal: there is no $z\in\C$ such that $\fgo(1/z)$ is a convergent series. However, choosing a  direction $\theta$  which is not anti-Stokes for   $\fgo$ allows one to evaluate $\Psi_\theta(z) = F(z)+  \fgo_{\theta} (1/z)$ at any $z$ in a large sector bisected by $\theta$. Here and below, 
such a direction will be said {\em not anti-Stokes for $\Psi$} and 
whenever we write $\fgo_{\theta}$ or $\Psi_\theta$ we shall assume implicitly that $\theta$ is not    anti-Stokes.

\medskip

Definition \ref{def:mixedfn} is a formal definition, but one may identify a mixed function with the holomorphic function it defines on a given large sector by means of the following lemma.

\begin{lem} \label{lemunicite} Let $\Psi$ be a mixed function, and $\theta\in\R$ be a non-anti-Stokes direction for $\Psi$. Then $\Psi_\theta$ is identically zero (as a holomorphic function on a large sector bisected by $\theta$) if, and only if, $\Psi$ is equal to zero (as a formal power series in $z$ and $1/z$).
\end{lem}

\medskip

Any mixed function $\Psi(z) = F(z)+\fgo(1/z)$ is solution of an $E$-operator. Indeed, 
this follows from applying \cite[Theorem 6.1]{andre}  twice: 
there exist   an $E$-operator $L$ such that $L( \fgo (1/z))=0$, and
an $E$-operator $M$ such that $M(L(F(z)))=0$ (because   $L(F(z))$ is an $E$-function). Hence $ML(F(z)+\fgo(1/z))=0$ and
 by \cite[p. 720, \S 4.1]{andre}, $ML$ is an $E$-operator.

\medskip

We formulate the following  conjecture, which implies both Conjecture \ref{conj3} and Conjecture~\ref{conjantie}.

\begin{conj}\label{conj1} Let $\Psi(z) $ be an mixed function, and $\theta\in (-\pi/2,\pi/2)$ be such that $\Psi_\theta(1)=0$.
Then $\frac{\Psi(z)}{z-1}$ is an mixed function.
\end{conj}

The conclusion of this conjecture is that $\Psi(z) = (z-1)   \Psi_1(z)$ for some mixed function $  \Psi_1$. This conclusion can be made more precise as follows; see \S\ref{ssec:proofprop1} for the proof.

\begin{prop}\label{prop1} Let $\Psi(z) = F(z)+\fgo(1/z)$ be an mixed function, and $\theta\in (-\pi/2,\pi/2)$ be such that $\Psi_\theta(1)=0$.
Assume that Conjecture \ref{conj1} holds for $\Psi$ and $\theta$. 

Then both $F(1)$ and $\fgo_\theta(1)$ are algebraic, and $\frac{  \fgo(1/z)  - \fgo_\theta(1) }{z-1} $ is an \antie-function.
\end{prop}
Of course, in the conclusion of this proposition, one may assert also that $\frac{  F(z)  -F(1) }{z-1} $  is an $E$-function using Theorem \ref{theoe}$(i)$.

Conjecture \ref{conj1} already has far reaching Diophantine consequences: Conjecture \ref{conjantie} and Theorem \ref{prop3} stated in the introduction, and also the following   result that encompasses Theorem \ref{theoantie} in the linear case.

\begin{theo}\label{theomixte} 
Assume that Conjecture \ref{conj1}  holds.

Let ${\bf \Psi}(z):={}^t(\Psi_1(z), \ldots, \Psi_n(z))$ be a vector of mixed functions  solution of a differential system ${\bf \Psi}'(z)=A(z){\bf \Psi}(z)$ for some matrix $A(z)\in M_n(\Qbar(z))$.  Let $\xi\in\Qbar\etoile$ and $\theta\in (\arg(\xi)-\pi/2,\arg(\xi)+\pi/2)$ ; assume that $\xi$ is not a pole of a coefficient of $A$, and that $\theta$ is anti-Stokes for none of the $\Psi_j$. 

Let $\lambda_1,\ldots, \lambda_n  \in\Qbar $ be  such that 
$$\sum_{i=1}^n \lambda_i 
 \Psi_{i,\theta}(\xi) =0.$$
Then there exist $L_1,\ldots, L_n \in \Qbar[z]$  such that 
$$\sum_{i=1}^n L_i(z) \Psi_i(z) = 0 \mbox{ identically  and } L_i(\xi) = \lambda_i  \mbox{ for any } i.$$
In particular, we have
$$\rk_{\Qbar}( \Psi_{1,\theta}(\xi), \ldots, \Psi_{n,\theta}(\xi)) = 
\rk_{\Qbar(z)}( \Psi_{1 }(z), \ldots, \Psi_{n}(z)). $$
\end{theo}

The proof of Theorem \ref{theomixte} follows exactly the linear part of the proof of Theorem \ref{theoantie} (see \S \ref{sec:proofsthm1}), which is based on \cite[\S 3]{beukers}. The only difference is that \antie-functions have to be replaced with mixed functions, and Conjecture \ref{conjantie} with Conjecture \ref{conj1}. In particular Proposition  \ref{propintermed} stated in \S \ref{sec:proofsthm1} remains valid with these modifications. 

However a product of mixed functions is not, in general, a mixed function. Therefore the end of  \cite[\S 3]{beukers} does not adapt to mixed functions, and there is no hope to obtain in this way a result on the transcendence degree of a field generated by values of mixed functions.

 As an application of Theorem~\ref{theomixte}, we can consider the mixed functions $1, e^{\beta z}$ and $\fgo(1/z):=\sum_{n=0}^\infty (-1)^n n! z^{-n}$, where $\beta$ is a fixed non-zero algebraic number. These three functions are linearly independent over $\mathbb C(z)$ and form a solution of a differential system with only 0 for singularity (because $(\fgo(1/z))'=(1+1/z)f(1/z)-1$), hence for any $\alpha\in \Qbar$, $\alpha>0$ and any $\varrho \in \Qbar^*$, the numbers $1, e^{\varrho}, \fgo_0(1/\alpha):=\int_0^\infty e^{-t}/(1+\alpha t)dt$ are $\Qbar$-linearly independent (for a fixed $\alpha$, take $\beta=\varrho/\alpha$).

\subsection{Values of mixed functions}\label{ssec:conseq2}

We denote by $\mixed$ the set of values $\Psi_\theta(1)$, where $\Psi$ is a mixed function and $\theta=0$  if it is not anti-Stokes, $\theta>0$ is sufficiently small otherwise. 
This set is obviously equal to ${\bf E}+{ \bf D}$.

\begin{prop} \label{propderigam}
For every integer $s\ge 0$ and every $a\in \mathbb Q^+$, $a\neq 0$, we have  $\Gamma^{(s)}(a)\in e^{-1} \mixed$.
\end{prop}

This results follows immediately from 
  Eq. \eqref{eq:gammasa} below (see   \S \ref{ssec:proofprop3}), written in the form
\begin{equation*}
    \Gamma^{(s)}(a)=e^{-1}\big((-1)^se s!  E_{a,s+1}(-1)+\fgo_{a,s+1;0}(1)\big),
\end{equation*}
because $e^zE_{a,s+1}(-z)$ is an $E$-function and $\fgo_{a,s+1;0}(1)$ is the 1-summation in the direction $0$ of an \tantie-function. 

\bigskip

It would be interesting to know if   $\Gamma^{(s)}(a) $ belongs to $ \mixed$. We did not succeed in proving it does, and we believe it does not. Indeed, for instance if we want to prove that $\gamma\in\mixed$, a natural strategy would be to construct an $E$-function $F(z) $ with asymptotic expansion of the form $\gamma+\log(z) + \fgo(1/z)$ in a large sector, and then to evaluate at $z=1$. However this strategy cannot work since there is no such $E$-function (see the footnote in the proof of Lemma \ref{lemunicite} in \S \ref{ssec:proofprop1}).

\subsection{Proofs concerning mixed functions}\label{ssec:proofprop1}

To begin with, let us take Proposition~\ref{prop1} for granted and prove that Conjecture \ref{conj1}   implies both Conjecture \ref{conj3} and Conjecture \ref{conjantie}. Concerning
Conjecture \ref{conjantie} it is clear. To prove that it implies Conjecture \ref{conj3}, let  $\xi\in {\bf D}$, i.e. $\xi=\fgo_{\theta}(1)$ is the 1-summation of an $\antie$-function $\fgo(z)$ in  the direction $\theta=0$ if it is not anti-Stokes, and $\theta>0 $ close to 0 otherwise.  Assume that $\xi$ is also in $\EE$:  we  have $\xi=F(1)$ for some $E$-function $F(z)$. Therefore, $\Psi(z)=F(z)-\fgo(1/z)$ is a mixed  function such that $\Psi_{\theta}(1)=0$. By Conjecture \ref{conj1} and Proposition~\ref{prop1}, we have $\xi=\fgo_{\theta}(1)\in \Qbar$. This concludes the proof that Conjecture \ref{conj1}   implies   Conjecture \ref{conj3}. 

\bigskip

Let us prove Proposition~\ref{prop1} now.  Assuming that Conjecture \ref{conj1} holds for $\Psi$ and $\theta$, there exists a mixed function $ \Psi_1(z)=F_1(z)+\fgo_1(1/z)$ such that $ \Psi(z) =(z-1)\Psi_1(z)$. 
We have
\begin{equation}\label{eqpreuvelem}
 F(z)-(z-1) F_1(z)+  \fgo (1/z) - (z-1)  \fgo_{1} (1/z)  = 0
 \end{equation}
as a formal power series in $z$ and $1/z$.
 Now notice that $z-1 = z(1-\frac1z)$, and  that we may assume   $\fgo $ and $\fgo_{1 }$ to have zero constant terms.
Denote by $\alpha $ the constant term of  $ \fgo (1/z) -  z(1-\frac1z) \fgo_{1}(1/z)$. Then we have 
$$  F(z)-(z-1) F_1(z)+ \alpha + \fgo_{2}(1/z)  =0$$
for some  \antie-function $\fgo_{2}$  without constant term, so that $\fgo_{2 } = 0$,  $ F(z) = (z-1) F_1(z)-  \alpha $ and $F(1) = -  \alpha\in\Qbar$.  This implies $\fgo_\theta(1) =   \alpha$, and  $\frac{  \fgo(1/z)  - \fgo_\theta(1) }{z-1}  =\fgo_{1 } (1/z) $ is an \antie-function since $\fgo_{2 } = 0$.
This  concludes the proof of Proposition \ref{prop1}.

\bigskip

At last, let us prove Lemma \ref{lemunicite}.
 We write  $\Psi(z) = F(z)+\fgo(1/z)$ and assume that $\Psi_\theta$ is identically zero.
Modifying $\theta$ slightly if necessary, we may assume that the asymptotic expansion $-\fgo(1/z)$ of $F(z)$ in a  large sector bisected by $\theta$ is given explicitly by \cite[Theorem~5]{ateo} applied to $F(z)-F(0)$; recall that such an asymptotic expansion is unique (see   \cite{ateo}). As in \cite{ateo} we let $g(z) = \sum_{n=1}^\infty a_n z^{-n-1}$ where the coefficients $a_n$ are given by $F(z)-F(0)= \sum_{n=1}^\infty \frac{a_n}{n!} z^n$.   For any $\sigma \in \C \setminus \{  0 \}$ there is no contribution in $e^{\sigma z}$ in the asymptotic expansion of $F(z)$, so that $g(z)$ is holomorphic at $\sigma$. At $\sigma=0$, the local expansion of $g$ is of the form $g(z) = h_1(z) +  h_2(z)\log(z)$ with $G$-functions $h_1$ and $h_2$, and  the coefficients of $h_2$ are related to those of $\fgo$; however we shall not use this special form~(\footnote{Actually we are proving that the asymptotic expansion of a non-polynomial $E$-function is never a $\C$-linear combination of functions $z^\alpha \log^k(z) \fgo(1/z)$ with $\alpha\in\Q$, $k\in\N$ and \antie-functions $\fgo$: some exponentials have to appear.}). Now recall that $g(z) = G(1/z)/z$ where $G$ is a $G$-function; then $G$ is entire and   has moderate growth at infinity (because $\infty$ is a regular singularity of $G$), so  it is a polynomial due to Liouville's theorem. This means that $F(z)$ is a polynomial in $z$. Recall that  asymptotic expansions in large sectors are unique. Therefore both $F$ and  $\fgo   $ are  constant functions, and $F+\fgo=0$. This concludes the proof of Lemma \ref{lemunicite}.

\section{Proof of Theorem \ref{theo:GpolyEantiE}: values of $G$-functions}\label{sec:proofthmpolyGEantiE}

In this section we prove Theorem \ref{theo:GpolyEantiE}. Let us begin with an example, starting with the relation proved in \cite[Proposition 1]{Michigan} for $z\in\C\setminus (-\infty,0]$: 
\begin{equation}\label{eq:gamma0701} 
    \gamma+\log(z)= z E_{1,2}(-z)-e^{-z} \fgo_{1,2;0}(1/z)
\end{equation}
where $E_{1,2}$ is an $E$-function,  and $\fgo_{1,2} $ is an \antie-function, both defined below in \S \ref{ssec:proofprop3}.

Apply Eq. \eqref{eq:gamma0701} at both $z$ and $2z$, and then substract one equation from the other. This provides a relation of the form 
\begin{equation} \label{eq:log(2)EantiE}
\log(2) = F(z) + e^{-z}  \fgo_{1;0} ( 1/z) + e^{-2z}  \fgo_{2;0} ( 1/z) 
\end{equation}
valid in a large sector bisected by $0$, with an $E$-function $F$ and \antie-functions $\fgo_1$ and $\fgo_2$. Choosing arbitrarily a positive real algebraic value of $z $ yields an explicit expression of $\log(2)\in\GG$ as a multivariate polynomial in elements of $\EE$ and  {\bf D}. But this example shows also that a polynomial in $E$- and \antie-functions may be constant eventhough there does not seem to be any obvious reason. In particular, 
 the functions $1$, $F(z)$, $ e^{-z}  \fgo_{1;0} ( 1/z)$, and $ e^{-2z}  \fgo_{2;0}  ( 1/z) $ are linearly dependent over $\C$. However we see no reason why they would be linearly dependent over $\Qbar$. This could be a major drawback to combine in  $E$- and \antie-functions, since functions that are linearly dependent over $\C$ but not over $\Qbar$ can not belong to any Picard-Vessiot extension  over $\Qbar$.   

\bigskip

Let us come now to the proof of Theorem \ref{theo:GpolyEantiE}. We first prove the second part, which runs as follows (it is reproduced from the unpublished note \cite{rivoalnote}). 

From the stability of the class of $E$-functions by $\frac{d}{dz}$ and $\int_0^z$, we deduce that the set of convergent integrals $\int_0^\infty F(x) dx$ of $E$-functions and the set of finite limits of $E$-functions along some direction as $z\to \infty$ are the same. Theorem~2$(iii)$ in \cite{ateo} implies that if an $E$-function has a finite limit as $z\to \infty$ along some direction, then this limit must be in ${\bf G}$. Conversely, let $\beta\in {\bf G}$. By Theorem 1 in \cite{gvalues}, there exists a $G$-function $G(z)=\sum_{n=0}^\infty a_n z^n$ of radius of convergence $\ge 2$ (say) such that $G(1)=\beta$. Let $F(z)=\sum_{n=0}^\infty \frac{a_n}{n!} z^n$ be the associated $E$-function. Then for any $z$ such that $\textup{Re}(z)>\frac12$, we have
$$
\frac{1}{z}G\Big(\frac{1}{z}\Big)=\int_0^{+\infty} e^{-xz}F(x)dx.
$$
Hence, $\beta=\int_0^{+\infty} e^{-x}F(x)dx$ where $e^{-z}F(z)$ is an $E$-function.

\medskip

We shall now prove the first part of Theorem \ref{theo:GpolyEantiE}. In fact, we shall prove a slightly more general result, namely Theorem \ref{th5} below.  We first recall a few notations. Denote by 
$\bfS$  the $\GG$-module generated by all derivatives $\Gamma^{(s)}(a)$ (with $s\in\N$ and $a\in\Q\setminus\Z_{\leq 0}$), and by $\VV$ the $\bfS$-module generated by $\EE$.  Recall that $\GG$, $\bfS$ and $\VV$ are rings. Conjecturally, $\GG = \calP[1/\pi]$ and $\VV =  \calP_e[1/\pi]$  where $\calP$ and $\calP_e$ are the ring of periods and the ring of exponential periods over $\Qbar$ respectively (see \cite[\S 2.2]{gvalues} and \cite[\S 4.3]{fuchsien}).   We have proved in \cite[Theorem 3]{fuchsien} that $\VV$ is the $\bfS$-module generated by the numbers $e^\rho \chi$, with $\rho\in\Qbar$ and $\chi\in $ {\bf D}.

\begin{theo}\label{th5}
The ring $\VV$ is the ring generated by $\EE$ and  {\bf D}. In particular, all  values of $G$-functions belong to the ring generated by $\EE$ and  {\bf D}. 
\end{theo}

In other words,  the elements of $\VV$ are exactly the sums of products $ab$ with $a\in\EE$ and $b \in  {\bf D}$.

\begin{proof}[Proof of Theorem \ref{th5}] 
We already know that $\VV$ is a ring, and that it contains $\EE$ and  {\bf D}. To prove the other inclusion, denote by $U$ the ring generated by $\EE$ and  {\bf D}. Using Proposition~\ref{propderigam} proved in \S \ref{ssec:conseq2} and the functional equation of $\Gamma$, we have $\Gamma^{(s)}(a) \in U $ for any  $s\in\N$ and any $a\in\Q\setminus\Z_{\leq 0}$.   Therefore for proving that $\VV\subset U$, it is enough to prove that $\GG  \subset U$.

Let $\xi\in \GG$. Using \cite[Theorem 3]{probsiegel} there exists an $E$-function $F(z)$
such that for any for any $\theta\in[-\pi,\pi)$ 
outside a finite set, $\xi$  is a coefficient of the asymptotic expansion of $F(z)$ in a large sector
bisected by  $\theta$. As the proof of  \cite[Theorem 3]{probsiegel}  shows, we can assume that $\xi $ is the coefficient of $e^z$ in this expansion. 

Denote by $L$ an $E$-operator of which $F$ is a solution, and by $\mu$ its order. André has proved \cite{andre} that there exists a basis $(H_1(z),\ldots,H_\mu(z))$ of formal solutions of $L$ at infinity such that for any $ j$, $e^{-\rho_j z}H_j(z)\in {\rm NGA}\{1/z\}_1^{\Qbar}$ for some algebraic number $\rho_j$. We recall that elements of $ {\rm NGA}\{1/z\}_1^{\Qbar}$ are arithmetic Nilsson-Gevrey series of order 1 with algebraic coefficients, i.e. $\Qbar$-linear combinations of functions $z^k (\log z)^\ell \fgo(1/z)$ with $k\in\Q$, $\ell\in\N$ and \antie-functions $\fgo$. Expanding   in this basis the asymptotic expansion of $F(z)$ in a large   sector
bisected by  $\theta$ (denoted by $\widetilde F$), there exist complex numbers $\kappa_1$, \ldots, $\kappa_d$ such that $\widetilde F(z)= \kappa_1 H_1(z) + \ldots +  \kappa_\mu H_\mu(z)$. Then we have $\xi =  \kappa_1  c_1 + \ldots +  \kappa_\mu c_\mu$, where $c_j$ is the coefficient of $e^z $ in $H_j(z)\in e^{ \rho_j z}{\rm NGA}\{1/z\}_1^{\Qbar}$. We have $c_j = 0$ if $\rho_j\neq 1$, and otherwise $c_j $ is the constant coefficient of $e^{-z}H_j(z)$: in both cases $c_j$ is an algebraic number. Therefore to conclude the proof that $\xi\in U$, it is enough to prove that $\kappa_1  ,  \ldots ,  \kappa_\mu \in U$. 

For simplicity let us prove that $\kappa_1\in U$. Given solutions $F_1,\ldots,F_\mu$ of $L$, we denote by $W( F_1,\ldots,F_\mu)$ the corresponding wronskian matrix. Then for any $z$ in a large sector
bisected by  $\theta$ we have 
$$ \kappa_1 = \frac{ \det W( F(z), H_{2,\theta}(z), \ldots,  H_{\mu,\theta}(z) )}{ \det W(   H_{1,\theta}(z), \ldots,  H_{\mu,\theta}(z) )}$$
where $H_{j,\theta}(z)$ is the 1-sommation of $H_j(z)$ in this sector. The determinant in the denominator belongs to $e^{a z} {\rm NGA}\{1/z\}_1^{\Qbar}$
with $a  = \rho_1+\ldots+\rho_\mu\in\Qbar$. As the proof of \cite[Theorem~6]{fuchsien} shows, there exist $b,c\in\Qbar$, with $c\neq 0$,  such that 
$$  \det W(   H_{1,\theta}(z), \ldots,  H_{\mu,\theta}(z) ) = c z^b e^{az} . $$
We take $z=1$, and choose $\theta=0$ if it is not anti-Stokes for $L$ (and $\theta>0$ sufficiently small otherwise). Then we have 
$$\kappa_1 = c^{-1} e^{-a} \Big(  \det W( F(z), H_{2,\theta}(z), \ldots,  H_{\mu,\theta}(z) ) \Big)_{| z=1}\in U.$$
This concludes the proof.
\end{proof}

\begin{Remark} The second part of Theorem \ref{theo:GpolyEantiE} suggests the following comments. It would be interesting to have a better understanding (in terms of {\bf E}, {\bf G} and {\bf D}) of the set of convergent integrals $\int_0^\infty R(x)F(x)dx$ where $R$ is a rational function in $\Qbar(x)$ and $F$ is an $E$-function, which are thus in {\bf G} when $R=1$ (see \cite{rivoalnote} for related considerations). Indeed, classical examples of such integrals are $\int_0^{+\infty} \frac{\cos(x)}{1+x^2} dx=\pi/(2e)  \in \pi {\bf E}$, Euler's constant $ \int_0^{+\infty} \frac{1-(1+x)e^{-x}}{x(1+x)} dx = \gamma \in {\bf E}+e^{-1}{\bf D}$ (using Eq.~\eqref{eq:gamma0701} and \cite[p.~248, Example~2]{WW}) and Gompertz constant $\delta:=\int_0^{+\infty} \frac{e^{-x}}{1+x} dx\in {\bf D}$. A large variety of behaviors can  thus be expected here. 

 For instance, using various explicit formulas in \cite[Chapters 6.5--6.7]{graryz}, it can be proved that
 $$
 \int_0^{+\infty} R(x)J_0(x)dx \in {\bf G}+ {\bf E}+ \gamma{\bf E}+\log(\Qbar^\star){\bf E}
 $$
 for any $R(x)\in \Qbar(x)$ without poles on $[0,+\infty)$, where $J_0(x)=\sum_{n=0}^{\infty} (ix/2)^{2n}/n!^2$ is a Bessel function.
 
A second class of examples is when $R(x)F(x)$ is an even function without poles on $[0,+\infty)$ and such that $\lim_{|x|\to \infty, \Im(x)\ge0} x^2 R(x)F(x)=0$. Then by the residue theorem,  
$$
\int_0^{+\infty} R(x)F(x)dx = i\pi\sum_{\rho, \, \Im(\rho)>0} \textup{Res}_{x=\rho}\big(R(x)F(x)\big)\in \pi {\bf E}
$$
where the summation is over the poles of $R$ in the upper half plane. 
\end{Remark}

\section{Derivatives of the $\Gamma $ function at rational points} \label{sec:consequences}

In this section we prove Theorem~\ref{prop3} and Proposition \ref{proppasdansE} stated in the introduction, dealing with $\Gamma^{(s)}(a)$. To begin with, we define $E$-functions $E_{a,s}(z)$ in \S \ref{ssec:conseq1} and prove a linear independence result concerning these functions. Then we prove in \S \ref{ssec:proofprop3} a formula for $\Gamma^{(s)}(a)$, namely Eq. \eqref{eq:gammasa}, involving $E_{a,s+1}(-1)$ and the 1-summation of an \antie-function. This enables us to prove Theorem~\ref{prop3} in \S \ref{subsec43} and Proposition \ref{proppasdansE} in \S \ref{subsec44}.

\subsection{Linear independence of a family of $E$-functions}\label{ssec:conseq1}

To study derivatives of the $\Gamma $ function at rational points, we need the following lemma. For $s\ge 1$ and  $a\in \mathbb Q\setminus \mathbb Z_{\leq 0}$, we consider the $E$-function $E_{a,s}(z):=\sum_{n=0}^\infty \frac{z^n}{n!(n+a)^s}$. 
\begin{lem}\label{lemmenew} 
$(i)$ For any $a\in \mathbb Q\setminus \mathbb Z$ and any $s\ge 1$,  the functions 
$$
1, e^z, e^zE_{a,1}(-z), e^zE_{a,2}(-z), \ldots, e^zE_{a,s}(-z)
$$ 
are linearly independent over $\mathbb C(z)$. 

\medskip

\noindent $(ii)$ For any $a\in \mathbb N^*$ and any $s\ge 2$,  the functions 
$$
1, e^z, e^zE_{a,2}(-z), \ldots, e^zE_{a,s}(-z)
$$ 
are linearly independent over $\mathbb C(z)$. 
\end{lem}
\begin{Remark} Part $(i)$ of the lemma is false if $a\in \mathbb N^{*}$ because $1, e^z, e^zE_{a,1}(-z)$ are $\mathbb Q(z)$-linearly dependent in this case (see the proof of Part $(ii)$ below).
\end{Remark}

\begin{proof} $(i)$ For simplicity, we set  $\psi_s(z):=e^zE_{a,s}(-z)$. We proceed by induction on $s\ge 1$. Let us first prove the case $s=1$. The derivative of $\psi_1(z)$ is $(1+(z-a)\psi_1(z))/z$. Let us assume the existence of a relation $\psi_1(z)=u(z)e^z+v(z)$ with $u,v\in \mathbb C(z)$ (a putative relation $U(z)+V(z)e^z+W(z)\psi_1(z)=0$ forces $W\neq 0$ because $e^z\notin \mathbb C(z)$). Then after differentiation of both sides, we end up with 
$$
\frac{1+(z-a)\psi_1(z)}{z}=\big(u(z)+u'(z)\big)e^z+v'(z).
$$
Hence, 
$$
\frac{1+(z-a)\big(u(z)e^z+v(z)\big)}{z}=\big(u(z)+u'(z)\big)e^z+v'(z).
$$
Since $e^z\notin \mathbb C(z)$, the function $v(z)$ is a rational solution of the differential equation $zv'(z)=(z-a)v(z)+1$: $v(z)$ cannot be identically 0, and it cannot be a polynomial (the degrees do not match on both sides). It must then have a pole at some point $\omega$, of order $d\ge 1$ say. We must have $\omega=0$ because otherwise the order of the pole at $\omega$ of $zv'(z)$ is $d+1$ while the order of the pole of $(z-a)v(z)+1$ is at most $d$. Writing $v(z)=\sum_{n\ge -d} v_nz^n$ with $v_{-d}\neq 0$ and comparing the term in $z^{-d}$ of $zv'(z)$ and $(z-a)v(z)+1$, we obtain that $d=a$. This forces $a$ to be an integer $\ge 1$, which is excluded. Hence, $1, e^z, e^zE_{a,1}(-z)$ are $\mathbb C(z)$-linearly independent.
\medskip

Let us now assume that the case $s-1\ge 1$ holds. Let us assume the existence of a relation over $\mathbb C(z)$
\begin{equation} \label{eq:psisz}
    \psi_s(z)=v(z)+u_0(z)e^z+\sum_{j=1}^{s-1} u_j(z)\psi_j(z).
\end{equation}
(A putative relation $V(z)+U_0(z)e^z+\sum_{j=1}^{s}U_j(z)\psi_j(z)=0$ forces $U_s\neq 0$ by the induction hypothesis). Differentiating \eqref{eq:psisz} and because 
 $\psi_j'(z)=(1-\frac{a}{z})\psi_j(z)+\frac{1}{z}\psi_{j-1}(z)$
for all $j\ge 1$ (where we have let $\psi_0(z)=1 $), we have 
\begin{multline} \label{eq:psisz2}
  A(z)\psi_s(z)+\frac{1}{z}\psi_{s-1}(z) =
   v'(z)+\big(u_0(z)+u'_0(z)\big)e^z
   +\sum_{j=1}^{s-1}u'_j(z)\psi_j(z)\\
+\sum_{j=1}^{s-1}u_j(z)\big(A(z)\psi_j(z)+\frac{1}{z}\psi_{j-1}(z)\big), 
\end{multline}
where $A(z):=1-a/z$. 
Substituting the right-hand side of \eqref{eq:psisz} for $\psi_s(z)$ on the left-hand side of \eqref{eq:psisz2}, we then deduce that
\begin{multline*}
    v'(z)-A(z)v(z)+\big(u'_0(z)
    +(1-A(z)) u_0(z) \big)e^z\\
    +\frac{1}{z}(z-a)u_1(z)\psi_1(z)+\sum_{j=1}^{s-1}u'_j(z)\psi_{j}(z)+\frac{1}{z}\sum_{j=1}^{s-1} u_j(z)\psi_{j-1}(z)
-\frac{1}{z}\psi_{s-1}(z) =0.
\end{multline*}
This is  a  non-trivial $\mathbb C(z)$-linear relation between 
$ 1, e^z, \psi_1(z), \psi_2(z), \ldots, \psi_{s-1}(z) $ because the coefficient of $\psi_{s-1}(z)$ is $u_{s-1}'(z)-1/z$ and it is  not identically 0 because $u_{s-1}'(z)$ cannot have a pole of order $1$. But by the induction hypothesis, we know that such a relation is impossible. 

\medskip

$(ii)$ The proof can be done by induction on $s\ge 2$ similarily. In the case $s=2$, assume the existence of a relation $\psi_2(z)=u(z)e^z+v(z)$ with $u(z), v(z)\in \mathbb C(z)$. By differentiation, we obtain
$$
\Big(1-\frac{a}{z}\Big)\psi_2(z)=-\frac{1}{z}\psi_1(z)+\big(u(z)+u'(z)\big)e^z+v'(z).
$$
By induction on $a\ge 1$, we have $\psi_1(z)=(a-1)!e^z/z^{a}+w(z)$ for some $w(z)\in \mathbb Q(z)$. Hence, we have 
$$
\Big(1-\frac{a}{z}\Big)u(z)=-\Big(\frac{(a-1)!}{z^{a+1}}+1\Big)u(z)+u'(z)
$$
which is not possible. Let us now assume that the case $s-1\ge 2$ holds, as well as the existence of a relation over $\mathbb C(z)$
\begin{equation} \label{eq:psisznew}
    \psi_s(z)=v(z)+u_0(z)e^z+\sum_{j=2}^{s-1} u_j(z)\psi_j(z).
\end{equation}
We proceed exactly as above by differentiation of both sides of \eqref{eq:psisznew}. Using the relation  $\psi_j'(z)=(1-\frac{a}{z})\psi_j(z)+\frac{1}{z}\psi_{j-1}(z)$
for all $j\ge 2$ and the fact that $\psi_1(z)=(a-1)!e^z/z^{a}+w(z)$, we obtain a relation 
$
\widetilde{v}(z)+\widetilde{u}_0(z)e^z+\sum_{j=2}^{s-1} \widetilde{u}_j(z)\psi_j(z) =0 
$
where $\widetilde{u}_{s-1}(z)=u_{s-1}'(z)-1/z$ 
cannot be identically 0. The induction hypothesis rules out the existence of such a relation.
\end{proof}

\subsection{A formula for $ \Gamma^{(s)}(a)$}\label{ssec:proofprop3}

Let $z>0$ and $a\in \mathbb Q^+$, $a\neq 0$. We have 
$$
\Gamma^{(s)}(a)=\int_0^\infty t^{a-1} \log(t)^s e^{-t}dt
=\int_0^z t^{a-1} \log(t)^s e^{-t}dt+ \int_z^\infty t^{a-1} \log(t)^s e^{-t}dt.
$$
On the one hand,
\begin{align*}
\int_0^z t^{a-1} \log(t)^s e^{-t}dt 
&= \sum_{n=0}^\infty \frac{(-1)^n}{n!}\int_0^z t^{a+n-1} \log(t)^s dt
\\
&=  \sum_{n=0}^\infty \frac{(-1)^n}{n!} \sum_{k=0}^s (-1)^k\frac{s!}{(s-k)!} \frac{z^{n+a}\log(z)^{s-k}}{(n+a)^{k+1}}
\\
&=\sum_{k=0}^s \frac{(-1)^k s!}{(s-k)!} z^{a}\log(z)^{s-k}E_{a,k+1}(-z);
\end{align*}
recall that $E_{a,j}(z) = \sum_{n=0}^\infty \frac{z^n}{n!(n+a)^j}$.
On the other hand, 
\begin{align*}
\int_z^\infty t^{a-1} \log(t)^s e^{-t}dt&=
e^{-z}\int_0^\infty (t+z)^{a-1} \log(t+z)^s e^{-t} dt
\\
&=z^{a-1}e^{-z}\sum_{k=0}^s \binom{s}{k}\log(z)^{s-k}\int_0^\infty (1+t/z)^{a-1} \log(1+t/z)^k e^{-t} dt.
\end{align*}
Now $z>0$ so that
$$
\fgo_{a,k+1;0}(z):=\int_0^\infty (1+tz)^{a-1} \log(1+tz)^k e^{-t} dt 
 = \frac1z \int_0^\infty (1+x)^{a-1} \log(1+x)^k e^{-x/z} dx 
$$
is the $1$-summation  at the origin  in the direction 0 of the $\antie$-function  
$$
\sum_{n=0}^\infty n! u_{a,k,n} z^{n},
$$
where the sequence $(u_{a,k,n})_{n\ge 0}\in \mathbb Q^{\mathbb N}$ is defined by the expansion of the $G$-function:
$$
(1+x)^{a-1} \log(1+x)^k = \sum_{n=0}^\infty u_{a,k,n} x^n.
$$
Note that if $k=0$ and $a\in \mathbb N^*$, then $u_{a,k,n} =0 $ for any $n\geq a$, and  $\fgo_{a,k+1;0}(1/z)$ is a polynomial in $1/z$. Therefore, we have for any $z>0$: 
\begin{equation*}
    \Gamma^{(s)}(a) =\sum_{k=0}^s\frac{(-1)^k s!}{(s-k)!} z^{a}\log(z)^{s-k}E_{a,k+1}(-z)
     +z^{a-1}e^{-z}\sum_{k=0}^s \binom{s}{k}\log(z)^{s-k}\fgo_{a,k+1;0}(1/z).
\end{equation*}
In particular, for $z=1$, this relation reads 
\begin{equation} \label{eq:gammasa}
    \Gamma^{(s)}(a)=(-1)^s s! E_{a,s+1}(-1)+e^{-1}\fgo_{a,s+1;0}(1).
\end{equation}
Since $\gamma=-\Gamma'(1)$ we obtain as a special case the formula
\begin{equation} \label{eq44bis}
\gamma=  E_{1,2}(-1)-e^{-1}\fgo_{1,2;0}(1),
\end{equation}
which is also a special case of Eq. \eqref{eq:gamma0701} proved in \cite{Michigan}.

\subsection{Proof of Theorem~\ref{prop3}} \label{subsec43}

 Let us assume that $\Gamma^{(s)}(a) \in \Qbar$ for some $a\in \mathbb Q^+\setminus \mathbb N$ and $s\ge 0$. Then $e^z\Gamma^{(s)}(a)+(-1)^{s+1} s! e^zE_{a,s+1}(-z)$ is an $E$-function. The relation $e\Gamma^{(s)}(a)+(-1)^{s+1} s!  eE_{a,s+1}(-1)=\fgo_{a,s+1;0}(1)$ proved at the end of \S \ref{ssec:proofprop3} shows that $\alpha:=e\Gamma^{(s)}(a)+(-1)^{s+1} s! eE_{a,s+1}(-1)\in \EE \cap {\bf D}$. Hence $\alpha$ is in $\Qbar$ by Conjecture~\ref{conj3} and we have a non-trivial $\Qbar$-linear relation between $1, e$ and $eE_{a,s+1}(-1)$: we claim that this is not possible. Indeed, consider the vector $$Y(z):={}^t(1, e^z, e^zE_{a,1}(-z), \ldots, e^zE_{a,s+1}(-z)).$$ 
It is solution of a differential system $Y'(z)=M(z)Y(z)$ where 0 is the only pole of $M(z)\in M_{s+3}(\Qbar(z))$ (see the computations in the proof of Lemma \ref{lemmenew} above). Since the components of $Y(z)$ are $\Qbar(z)$-linearly independent by Lemma \ref{lemmenew}$(i)$, we deduce from Beukers' \cite[Corollary 1.4]{beukers} that $$1,\, e, \, eE_{a,1}(-1),\, \ldots, \,eE_{a,s+1}(-1)$$ are $\Qbar$-linearly independent, and in particular that $1, e$ and $eE_{a,s+1}(-1)$ are $\Qbar$-linearly independent. This concludes the proof if $a\in \mathbb Q^+\setminus \mathbb N$.

\medskip

Let us assume now  that $\Gamma^{(s)}(a)\in \Qbar$  for some $a\in \mathbb N^*$ and $s\ge 1$.  Then $e^z\Gamma^{(s)}(a)+(-1)^{s+1} s! e^zE_{a,s+1}(-z)$ is an $E$-function. The relation $\Gamma^{(s)}(a)+(-1)^{s+1} s! E_{a,s+1}(-1)=e^{-1}\fgo_{a,s+1;0}(1)$ shows that $\alpha:=e\Gamma^{(s)}(a)+(-1)^{s+1} s!e E_{a,s+1}(-1)\in \EE \cap {\bf D}$. Hence $\alpha$ is in $\Qbar$ by Conjecture \ref{conj3} and we have a non-trivial $\Qbar$-linear relation between $1, e$ and $eE_{a,s+1}(-1)$: we claim that this is not possible. Indeed, consider the vector $Y(z):={}^t(1, e^z, e^zE_{a,2}(-z), \ldots,$ $e^zE_{a,s+1}(-z))$: it is solution of a  differential system $Y'(z)=M(z)Y(z)$ where 0 is the only pole of $M(z)\in M_{s+2}(\Qbar(z))$. Since the components of $Y(z)$ are $\Qbar(z)$-linearly independent by Lemma \ref{lemmenew}$(ii)$, we deduce again from Beukers' theorem that $$1, \,e, \, eE_{a,2}(-1), \,\ldots, \, eE_{a,s+1}(-1)$$ are $\Qbar$-linearly independent, and in particular that $1, e$ and $eE_{a,s+1}(-1)$ are $\Qbar$-linearly independent.  This concludes the proof of Theorem~\ref{prop3}.
 
\subsection{Proof of Proposition~\ref{proppasdansE}}\label{subsec44}

 Recall that Eq. \eqref{eq44bis} proved in \S \ref{ssec:proofprop3} reads 
$
 eE_{1,2}(-1)- e\gamma=\fgo_{1,2;0}(1).
$
Assuming that $\gamma \in \EE$, the left-hand side is in $\EE$ while the right-hand side is in {\bf D}. Hence both sides are in $\Qbar$ by Conjecture~\ref{conj3}. Note that, by integration by parts,  
$$
\fgo_{1,2;0}(1) =\int_0^\infty \log(1+t)e^{-t}dt = \int_0^{\infty} \frac{e^{-t}}{1+t} dt
$$
is Gompertz's constant. 
Hence, by Corollary~\ref{coro:0701}  (which holds under Conjecture~\ref{conjantie}), the number $\fgo_{1,2;0}(1)$ is not in $\Qbar$. Consequently, $\gamma\notin \EE$.

\medskip

Similarly, Eq.  \eqref{eq:gammasa} with $a\in \mathbb Q\setminus \mathbb Z$ and $s=0$ reads 
$
e\Gamma(a)-eE_{a,1}(-1)=\fgo_{a,1;0}(1).$
Assuming that $\Gamma(a) \in \EE$, the left-hand side is in $\EE$ while the right-hand side is in {\bf D}. Hence both sides are in $\Qbar$ by Conjecture~\ref{conj3}. But by Corollary \ref{coro:0701} (which holds under Conjecture~\ref{conjantie}), the number $\fgo_{a,1;0}(1) = \int_0^\infty  (1+t)^{a-1} e^{-t}dt $ is not in $\Qbar$. Hence, $\Gamma(a)\notin \EE$.

\section{Application of Beukers' method and consequence} \label{sec5}

In this section we prove  Theorem \ref{theoantie} and Corollary \ref{coro:0701} stated in the introduction.

\subsection{Proof of Theorem \ref{theoantie}}\label{sec:proofsthm1}

The proof of Theorem \ref{theoantie} is based on the arguments given in \cite{beukers}, except that $E$-functions have to be replaced with \antie-functions, and 1-summation in non-anti-Stokes directions is used for evaluations. Conjecture \ref{conjantie} is used as a substitute for Theorem \ref{theoe}$(i)$.

The main step is the following result, the proof of which is analogous to the end of the proof of \cite[Corollary 2.2]{beukers}.

\begin{prop}\label{propintermed} 
Assume that Conjecture \ref{conjantie}  holds.

Let   $\fgo$ be an \antie-function,  $\xi\in\Qbar\etoile$ and $\theta\in (\arg(\xi)-\pi/2,\arg(\xi)+\pi/2)$. Assume that   $\theta$ is not anti-Stokes for   $\fgo$, and that $\fgo_\theta(1/\xi)=0$. Denote by $Ly=0$ a differential equation, of minimal order, satisfied by $\fgo(1/z)$. 

Then all solutions of $Ly=0$ are holomorphic and vanish at $\xi$; the differential operator $L$ has an apparent singularity at $\xi$.
\end{prop}

To deduce  Theorem \ref{theoantie} from Proposition \ref{propintermed}, it is enough to follow \cite[\S 3]{beukers}.

\subsection{Proof of Corollary \ref{coro:0701}}\label{ssec:proofcoro1}

Let $s\in \Q\setminus \mathbb Z_{\ge 0}$.
The \tantie-function $\fgo(z):=\sum_{n=0}^\infty  s(s-1)\ldots (s-n+1) z^n$ is solution of the inhomogeneous differential equation $z^2\fgo'(z)+(1-sz)\fgo(z)-1=0$, which can be immediately transformed into a differential system satisfied by the vector of \tantie-functions ${}^t(1, \fgo(z))$.  The coefficients of the matrix have only 0 as pole. Moreover, $\fgo(z)$ is a transcendental function because $s\notin \mathbb Z_{\ge 0}$. Hence, by Theorem~\ref{theoantie}, $\fgo_0(1/\alpha) \notin\Qbar$ when  
$\alpha\in \Qbar$, $\alpha>0$, because $0$ is not an anti-Stokes direction of $\fgo(z)$. It remains to observe that this 1-sommation is 
$$
%\fgo_0(z) = 
\int_0^\infty (1+tz)^s e^{-t}dt.
$$

\medskip

\noindent S. Fischler,  Universit\'e Paris-Saclay,  CNRS, Laboratoire de math\'ematiques d'Orsay,  91405 Orsay, France.

\medskip

\noindent T. Rivoal, Universit\'e Grenoble Alpes, CNRS, Institut Fourier, CS 40700, 38058 Grenoble cedex 9, France.

\bigskip

\noindent Keywords: $E$-functions, \antie-functions, $G$-functions, Gamma function, Siegel-Shidlovskii Theorem.
\medskip

\noindent MSC 2020: 11J91 (Primary), 33B15 (Secondary)

\end{document}